\documentclass[11pt]{article}
\usepackage[T1]{fontenc}
\usepackage{amsfonts}
\usepackage{amsmath}
\usepackage{color}
\usepackage{amssymb}
\usepackage{fancyhdr}
\RequirePackage{theorem}
\RequirePackage{psfrag}
\RequirePackage[dvips]{graphicx}

\newtheorem{thm}{\bf Theorem}[section]
\newtheorem{lem}[thm]{\bf Lemma}

\newtheorem{coro}[thm]{\bf Corollary}
\newtheorem{Def}{\bf Definition}[section]

\newtheorem{Hyps}{\bf Assumptions}
\newtheorem{rmk}{\bf Remark}[section]

\newcommand{\CQFD}{\hfill $\blacksquare$}
\newenvironment{proof}{{\bf Proof}\\}{\CQFD}
\newcommand{\R}{\mathbb{R}}
\newcommand{\N}{\mathbb{N}}

\newcommand{\Dd}{\mathcal{D}}

\newcommand{\llbrack}{\rm [ \! [}
\newcommand{\rrbrack}{\rm ] \! ]}
\newcommand{\1}{\rm 1 \! l}

\newcommand{\Qq}{\mathcal{Q}}
\newcommand{\eps}{\varepsilon}

\def\ds{\displaystyle}
\def\IN{\llbrack 1,N\rrbrack}
\def\a{\alpha}
\def\b{\beta}
\def\t{\tilde}
\def\l{\lambda}
\def\phii{\varphi_i}
\def\phij{\varphi_j}
\def\phiin{\varphi_{i,n}}

\def\O{\Omega}
\def\OT{\Omega\times(0,T)}
\def\OiT{\Omega_i\times(0,T)}
\def\G{\Gamma}
\def\Gij{{\Gamma_{i,j}}}
\def\GijT{{\Gamma_{i,j}\times(0,T)}}
\def\grad{\nabla}
\def\div{\nabla\cdot} 
\def\SiN{\sum_{i=1}^N}
\def\h{\hat}
\def\Oo{\mathcal{O}}

\begin{document}

\title{Two-phase flows involving capillary barriers in heterogeneous porous media.}

\date{}

\author{Cl\'ement {\sc Canc\`es} (corresponding author), \\
Universit\'e de Provence, \\
39, rue F. Joliot Curie, 13453 Marseille Cedex 13. \\
%\courrier{
cances@cmi.univ-mrs.fr
%} 
\\[10pt]
Thierry {\sc Gallou\"et}, \\
Universit\'e de Provence, \\
39, rue F. Joliot Curie, 13453 Marseille Cedex 13. \\
%\courrier{
gallouet@cmi.univ-mrs.fr
%}
 \\[10pt]
Alessio \sc{Porretta}, \\
Universit\`a di Roma Tor Vergata, \\
Via della Ricerca Scientifica 1, 00133 Roma. \\
%\courrier{
porretta@mat.uniroma2.it }
%}

\maketitle
\abstract{
We consider a simplified model of a two-phase flow through a heterogeneous 
porous medium, in which the convection is neglected.
This leads to a nonlinear degenerate parabolic problem in a domain shared in an arbitrary finite number of homogeneous porous media. We introduce a new way to connect capillary pressures on the interfaces between the 
homogeneous domains, which leads to a general notion of solution. We then compare this notion of solution with an existing one, showing that it allows to deal with 
a larger class of problems. We prove the existence of such a solution in a general 
case, then we prove the existence and the uniqueness of a regular solution in the 
one-dimensional case for regular enough initial data. 
}\\[20pt]

\section{Presentation of the problem}

The models of immiscible two-phase flows are widely used in petroleum 
engineering, particularly in basin modeling, whose aim can be the prediction 
of the migration of hydrocarbon components  at geological time scale in a 
sedimentary basin. 

The heterogeneousness of the porous medium leads to the phenomena of 
oil-trapping and oil-expulsion, which is modeled with discontinuous capillary 
pressures between the different geological layers.  

The physical principles models and the mathematical models can be found in 
\cite{AS79,Bear72,CJ86,vDMdN95,Ench}. The phenomenon of capillary trapping has 
been completed only in simplified cases (see \cite{BPvD03}), and several
numerical methods have been developed (see e.g. \cite{EEN98,EEM06}).

The aim of this paper is to  introduce a new notion of weak solution,  which allows 
us to deal with  more general cases than those treated  in \cite{EEM06}, while it is equivalent 
to the notion of weak solution 
introduced in \cite{EEM06} on the already treated cases. 
We will consider a simplified model \eqref{P} defined page \pageref{P}, in which the convection is neglected, 

We then give a uniqueness 
result in the one dimensional case which is  inspired from  the result in \cite{BPvD03}  
and extends this latter one to more general situations, by requiring weaker assumptions 
on the solutions and applying to  a larger class of initial data.

We have to make some assumptions on the heterogeneous porous medium:
\begin{Hyps}\label{geom}{\bf (Geometrical assumptions)}
\begin{enumerate}
\item The heterogeneous porous medium is represented by a polygonal bounded 
connected domain 
$\O\subset\R^d$ with $meas_{\R^d}(\O)>0$, where $meas_{\R^n}$ is the Lebesgue's 
measure of $\R^n$.
\item There exists a finite number $N$ of polygonal connected subdomains 
$(\O_i)_{1\le i\le N}$ of $\O$ such that:
\begin{enumerate}
\item for all $i\in\IN$, $meas_{\R^d}(\O_i)>0$,
\item $\ds \bigcup_{i=1}^N \overline\O_i=\overline \O$,
\item for $(i,j)\in\IN^2$ with $i\neq j$, $\O_i\cap\O_j=\emptyset$.
\end{enumerate}
Each $\O_i$ represents an homogeneous porous medium. One denotes, for all 
$(i,j)\in\IN^2$, $\Gij\subset\O$ the interface between the geological layers 
$\O_i$ and $\O_j$, defined by $\overline\G_{ij}=\partial\O_i\cap\partial\O_j$.
\end{enumerate}
\end{Hyps}
\begin{figure}[htb]
\centering
\includegraphics[scale=.4]{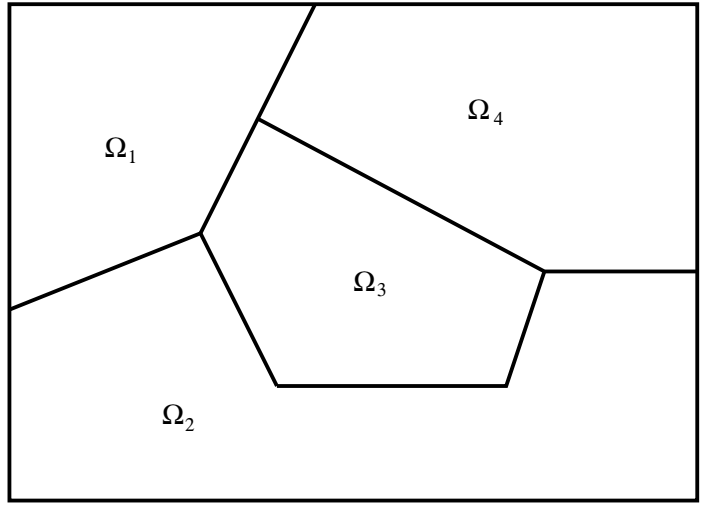}
\caption{An example for the domain $\Omega$ }
\label{domaine} 
\end{figure}
We consider an incompressible and immiscible oil-water flow through $\O$,  
and thus through each $\O_i$. Using Darcy's law, the conservation of oil and 
water phases is given for all $(x,t)\in\OiT$,
\begin{equation}\label{modele1}
\left\{
\begin{array}{l}
\ds \phi_i\partial_t u_i(x,t) -\div\big{(}\eta_{o,i}(u_i(x,t))
(\grad p_{o,i}(x,t)-\rho_o {\bf g})\big{)}=0,\\
-\ds \phi_i\partial_t u_i(x,t) -\div\big{(}\eta_{w,i}(u_i(x,t))
(\grad p_{w,i}(x,t)-\rho_w {\bf g})\big{)}=0,\\
p_{o,i}(x,t)-p_{w,i}(x,t)=\pi_i(u_i(x,t)),
\end{array}
\right.\end{equation}
where $u_i\in[0,1]$ is the oil saturation in $\O_i$ 
(and therefore $1-u_i$ the water saturation), $\phi_i\in\ ]0,1[$ 
is the porosity of $\O_i$, which is supposed to be constant in each $\O_i$ 
for the sake of simplicity, $\pi_i(u_i(x,t))$ is the capillary pressure, 
and ${\bf g}$ is the gravity acceleration. The indices $o$ and $w$ 
respectively stand for the oil and the water phase. Thus, for $\sigma=o,w$, 
$p_{\sigma,i}$ is the pressure of the phase $\sigma$, $\eta_{\sigma,i}$ 
is the mobility of the phase $\sigma$, and $\rho_\sigma$ is the density of 
the phase $\sigma$.

We have now to make assumptions on the data to explicit the transmission 
conditions through the interfaces $\Gij$:
 \begin{Hyps}\label{pii_li} {\bf (Assumptions on the data)}
 \begin{enumerate}
 \item for all $i\in\IN$, $\pi_i\in C^1([0,1],\R)$, with $\pi'_i(x)>0$ for 
 $x\in]0,1[$, 
\item  for all $i\in\IN$, $\eta_{o,i}\in C^0([0,1],\R_+)$ is an increasing 
function fulfilling $\eta_{o,i}(0)=0$, 
\item  for all $i\in\IN$, $\eta_{w,i}\in C^0([0,1],\R_+)$ is a decreasing 
function fulfilling $\eta_{w,i}(1)=0$,
\item the initial data $u_0$ belongs to $L^\infty(\O)$, $0\le u_0 \le 1$.
 \end{enumerate}
\end{Hyps}

One denotes $\a_i=\lim_{s\rightarrow0}\pi_i(s)$ and $\b_i=\lim_{s\rightarrow1}
\pi_i(s)$.
We can now define the monotonous graphs $\tilde\pi_i$ by:
\begin{equation}\label{tpii}
\t\pi_i(s)=\left\{\begin{array}{ll}
\ds \pi_i(s)  &\text{ if } s\in]0,1[,\\
\ds ]-\infty,\a_i]& \text{ if } s=0,\\
\ds [\b_i,+\infty[ &\text{ if } s=1.
\end{array}\right.
\end{equation}
\begin{figure}[htb]
\centering
\includegraphics[scale=.4]{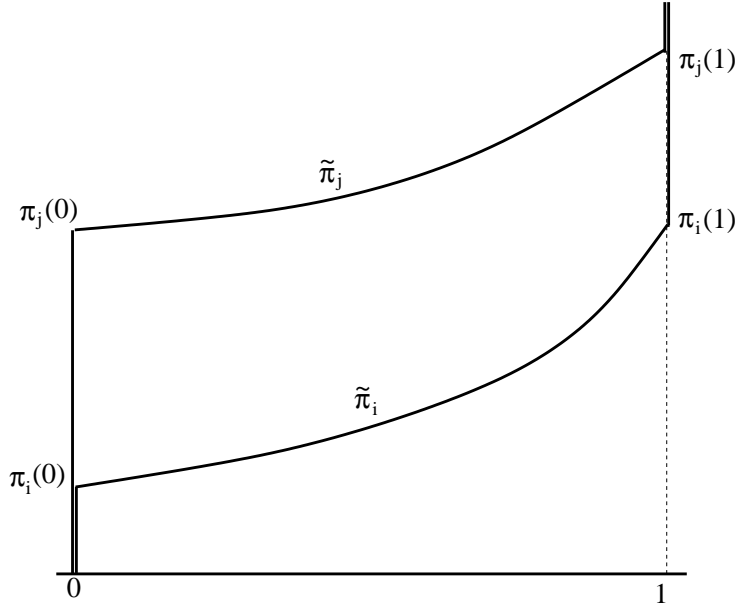}
\caption{Graphs for the capillary pressures}
\label{pression_nb} 
\end{figure}
As it is exposed in \cite{EEM06}, the following conditions must be satisfied on the traces of $u_i$, 
$p_{\sigma,i}$ and $\grad p_{\sigma,i}$ on $\GijT$, still denoted respectively $u_i$, $p_{\sigma,i}$ and 
$\grad p_{\sigma,i}$ (see \cite{Bear72}):
\begin{enumerate}
\item for any $\sigma=o,w$, $(i,j)\in \IN^2$ such that $\Gij\neq\emptyset$, the flux of the phase $\sigma$ 
through $\Gij$ must be continuous:
\begin{equation}\label{flux}
\eta_{\sigma,i}(u_i)(\grad p_{\sigma,i}-\rho_\sigma{\bf g})\cdot{\bf n}_i+
\eta_{\sigma,j}(u_j)(\grad p_{\sigma,j}-\rho_\sigma{\bf g})\cdot{\bf n}_j=0,
\end{equation}
where ${\bf n}_i$ denotes the outward normal of $\Gij$ to $\O_i$;
\item for any $\sigma=o,w$, $(i,j)\in \IN^2$ such that $\Gij\neq\emptyset$, either $p_\sigma$ is continuous or 
$\eta_\sigma=0$. Since the saturation is itself discontinuous across $\Gij$, one must express the mobility at 
the upstream side of the interface. This gives
\begin{equation}\label{raccord_pressions_partielle}
\eta_{\sigma,i}(u_i)(p_{\sigma,i}-p_{\sigma,j})^+-\eta_{\sigma,j}(u_j)(p_{\sigma,j}-p_{\sigma,i})^+=0.
\end{equation}
\end{enumerate}

The conditions (\ref{raccord_pressions_partielle}) have direct consequences on the behaviour of the capillary pressures 
on both side of $\Gij$. Indeed, if $0<u_i,u_j<1$, then the partial pressures $p_o$ and $p_w$ have both to be continuous, 
and so we have the connection of the capillary pressures $\pi_i(u_i)=\pi_j(u_j)$. If $u_i=0$ and $0<u_j<1$, then 
$p_{o,i} \ge p_{o,j}$ and $p_{w,i}=p_{w,j}$, thus $\pi_j(u_j)\le \pi_i(0)$. The same way, $u_i=1$ and $0<u_j<1$ 
implies $\pi_j(u_j)\ge \pi_i(1)$. If $u_i=0$, $u_j=1$, then $p_{o,i}\ge p_{o,j}$ and $p_{w,i}\le p_{w,j}$, 
so $\pi_i(0)\ge\pi_j(1)$. 
Checking that the definition of the graphs $\t\pi_i$ and $\t\pi_j$ implies $\t\pi_i(0)\cap\t\pi_j(0)\neq\emptyset$, 
$\t\pi_i(1)\cap\t\pi_j(1)\neq\emptyset$, we can claim that (\ref{raccord_pressions_partielle}) leads to: 
\begin{equation}\label{raccord_pi1}
\t\pi_i(u_i)\cap\t\pi_j(u_j)\neq\emptyset.
\end{equation}
We introduce the global pressure in $\O_i$
\begin{equation}\label{pression_globale}
\overline{p}_i (x,t)=p_{w,i}(x,t)+\int_0^{u_i(x,t)}\frac{\eta_{o,i}(a)}{\eta_{o,i}(a)+\eta_{w,i}(a)}
\pi'_i(a)\text{d}a
\end{equation}
(see e.g. \cite{AKM90} or \cite{CJ86}), and the global mobility in $\O_i$
\begin{equation}\label{mobilite_globale}
\l_i(u_i(x,t))=\frac{\eta_{o,i}(u_i(x,t))\eta_{w,i}(u_i(x,t))}{\eta_{o,i}(u_i(x,t))+\eta_{w,i}(u_i(x,t))}
\end{equation}
which verifies $\l_i(0)=\l_i(1)=0$, and $\l_i(s)>0$ for $0<s<1$.
Taking into account (\ref{pression_globale}) and (\ref{mobilite_globale}) in (\ref{modele1}), and adding the conservation laws leads to, for $(x,t)\in\OiT$:
\begin{equation}\label{systeme_complet}
\left\{
\begin{array}{l}
\ds\phi_i\partial_t u_i(x,t) - \div\big{(}\eta_{o,i}(u_i(x,t))(\grad\overline{p}_i(x,t)-\rho_o {\bf g})-\l_i(u_i(x,t)) \grad \pi_i(u_i(x,t))\big{)}=0,\\
\ds -\div\left(\sum_{\sigma=o,w}\eta_{\sigma,i}(u_i(x,t))(\grad\overline{p}_i(x,t)-\rho_\sigma {\bf g})\right)=0.
\end{array}
\right.
\end{equation}

We neglect the convective effects, so that we focus on the mathematical modeling of flows with 
discontinuous capillary pressures, which seem to necessary to explain the phenomena of oil trapping. 
This simplification will allow us to neglect the coupling with the second equation of 
(\ref{systeme_complet}), and we get the simple degenerated parabolic equation in $\OiT$:
%We neglect the convective effects, because the term $\l_i(u_i(x,t)) \grad \pi_i(u_i(x,t))$ in the first equation of (\ref{systeme_complet}) seems to be sufficient to explain the phenomena of oil-trapping and oil-expulsion. 
\begin{equation}\label{parabolic}
\phi_i\partial_t u_i(x,t) -\div(\l_i(u_i(x,t))\grad\pi_i(u_i(x,t)))=0\quad \text{ in }\OiT.
\end{equation}
In this simplified framework, the transmission condition \eqref{flux} on the fluxes through $\Gij$ can be rewritten:
\begin{equation}\label{flux_simples}
\l_i(u_i(x,t))\grad(\pi_i(u_i(x,t)))\cdot {\bf n}_i+\l_j(u_j(x,t))\grad(\pi_j(u_j(x,t)))\cdot {\bf n}_j=0 \quad \text{on }\GijT.
\end{equation}
We suppose furthermore that $u_i(x,0)=u_0(x)$ for $x\in \O_i$. In the remainder of this paper, we suppose to take a homogeneous Neumann boundary condition, 
The existence of a weak solution proven in section \ref{existence} can be extended 
to the case of non-homogeneous Dirichlet conditions. Nevertheless, 
homogeneous Neumann boundary conditions are needed to prove the 
theorem \ref{regularity_graph}, and thus to prove the conclusion theorem \ref{existence_SOLA}

Taking into account the equations (\ref{raccord_pi1}), (\ref{parabolic}), (\ref{flux_simples}), the boundary condition, and the initial condition, we can write the problem we aim to solve this way: for all $i\in\IN$, for all $j\in\IN$ such that $\Gij\neq\emptyset$, 
\begin{equation}\label{P}\tag{$\mathcal{P}$}
\left\{\begin{array}{ll}
\phi_i\partial_t u_i -\div(\l_i(u_i)\grad\pi_i(u_i))=0 &\text{in }\OiT,\\
\t\pi_i(u_i)\cap\t\pi_j(u_j)\neq\emptyset & \text{on }\GijT,\\
\l_i(u_i)\grad(\pi_i(u_i))\cdot {\bf n}_i+\l_j(u_j)\grad(\pi_j(u_j))\cdot {\bf n}_j=0 & \text{on }\GijT,\\
\l_i(u_i)\grad(\pi_i(u_i))\cdot {\bf n}_i=0 &\text{on }\partial\O_i\cap\partial\OT,\\
u_i(\cdot,0)=u_0(x) &\text{in }\O_i.
\end{array}\right.
\end{equation}

\begin{rmk}
All the results presented in this paper still hold if one not neglects the effect of the gravity and if one assumes that the global pressure is known, that is for 
problems of the type :
$$
\left\{
\begin{array}{ll}
\ds\phi_i \partial_t u_i  + \div \left( {\bf q} f_i(u_i) + \l_i(u_i) (\rho_o - \rho_w) {\bf g} 
\ds- \l_i(u_i) \grad \pi_i(u_i) \right) =0  &\textrm{in }\OiT, \\
\ds\t\pi_i(u_i)\cap\t\pi_j(u_j)\neq\emptyset & \text{on }\GijT,  \\
\ds \sum_{k=i,j}\left( {\bf q} f_k(u_k) + \l_k(u_k) (\rho_o - \rho_w) {\bf g} 
- \l_k(u_k) \grad \pi_k(u_k) \right)\cdot{\bf n}_k = 0  &\text{on }\GijT, \\
\ds \left( {\bf q} f_i(u_i) + \l_i(u_i) (\rho_o - \rho_w) {\bf g} 
- \l_i(u_i) \grad \pi_i(u_i) \right)\cdot{\bf n}_i =0 &\text{on }\partial\O_i\cap\partial\OT,\\
\ds u_i(\cdot,0)=u_0(x) &\text{in }\O_i,
\end{array}\right.
$$
where $f_i$ is supposed to be a $C^1([0,1],\R)$-increasing function, $\l_i$ is also supposed to belong to
$C^1([0,1],\R_+)$  and ${\bf q}$ satisfies 
\begin{itemize}
\item $\forall i$, ${\bf q}\in \left( C^1(\overline\O_i\times[0,T])\right)^d$,
\item $\div{\bf q} =0$ in $\OiT$, 
\item ${\bf q}_{|\O_i}\cdot{\bf n}_i+ {\bf q}_{|\O_j}\cdot{\bf n}_j =0$ on $\GijT$,
\item ${\bf q}\cdot {\bf n} =0$.
\end{itemize}
In order to ensure the uniqueness result stated in theorem \ref{prop_unicite}, the technical condition 
(see \cite{AL83} or \cite{Otto96}) has to be fulfilled: 
$$\forall i,\qquad f_i\circ\varphi_i^{-1}, \l_i\circ\varphi_i^{-1} \in C^{0,1/2}([0,\phii(1)],\R).$$ 
\end{rmk}
\begin{rmk}
In the modeling of two-phase flows, irreducible saturations are often taken into account. 
One can suppose that there exists $s_i$ and $S_i$ ($0<s_i<S_i<1$) such that $\l_i(s)=0$ if 
$s\notin (s_i,S_i)$. In such a case, the problem \eqref{P} becomes strongly degenerated, 
but a convenient scaling eliminates this difficulty (at least if $s_i\le u_0 \le S_i$ a.e. in $\O_i$).
Moreover, the dependance of the capillary pressure with regard to the saturation can be weak, 
at least for saturations not too close to $0$ or $1$. Thus the effects of the capillarity are often 
neglected for the study of flows in homogeneous porous media, leading to the Buckley-Leverett 
equation (see e.g. \cite{GMT96}). Looking for degeneracy of $u\mapsto\pi_i(u)$ is a more complex 
problem, particularly if the convection is not neglected as above. 
Suppose for example that $\pi_i(u) = \eps u+ P_i$, where $P_i$ are constants, and let $\eps$ 
tend $0$.
Non-classical shocks can appear at the level of the interfaces $\Gij$ (see \cite{non-classic}). 
Thus the notion of entropy solution used by Adimurthi, J. Jaffr{\'e}, and G.D. Veerappa Gowda 
\cite{AJV03} is not sufficient to deal with this problem. 
This difficulty has to be overcome to consider degenerate parabolic problem. But it seems clear 
that  the notion of entropy solution developed by 
K.H. Karlsen, N.H. Risebro, J.D. Towers 
\cite{KRT02a, KRT02b,KRT03} is not  adapted to our problem. 
\end{rmk}

% 2) LA NOTION DE SOLUTION GRAPHE
\section{The notion of weak solution}\label{notion}
In this section, we introduce the notion of weak solution to the 
problem~(\ref{P}), which is more general than the notion of weak solution 
given in \cite{Ench,EEM06}. Indeed, we are able to define such a  
solution even in the case of an arbitrary finite number of different 
homogeneous porous media. Furthermore, the notion of weak solution introduced
in this paper is still available in cases where the one defined in \cite{EEM06} has no more 
sense.
%
%in cases where the truncated capillary pressures have no 
%consistence. 
We finally show that the two notions of solution are equivalent 
in the case where the notion of weak solution in the sense of \cite{EEM06} is 
well defined. The existence of a weak solution to problem~(\ref{P}) 
in a wider case is the aim of the section~\ref{existence}.

%======== introduire un graphe de \t\pi_i======================
One denotes by $\phii$ the $C^1([0,1],\R_+)$ function which naturally appears 
in the problem~(\ref{P}) and which is defined by: $\forall s\in[0,1]$,
\begin{equation}\label{phii_graph}
\phii(s)=\int_0^s \l_i(a)\pi'_i(a)da.
\end{equation}

\begin{rmk}\label{phii_invers}
The assumptions on the data insure that $\phii'>0$ on $]0,1[$, and so we can 
define an increasing continuous function $\phii^{-1}: 
[0,\phii(1)]\rightarrow [0,1]$.
\end{rmk}

We are now able to define the notion of weak solution to the problem~(\ref{P}).

\begin{Def}[weak solution to the problem~(\ref{P})]\label{gws}
Under assumptions \ref{geom} and \ref{pii_li}, a function $u$ is said to be a 
 weak solution to the problem~(\ref{P})
if it verifies:
\begin{enumerate}
\item $u\in L^\infty(\OT)), 0\le u\le 1 \text{ a.e. in }\OT$,
\item $\forall i\in\IN$, $\phii(u_i)\in L^2(0,T;H^1(\O_i))$, where $u_i$ denotes the restriction of $u$ to $\OiT$,
\item $\t\pi_i(u_i)\cap\t\pi_j(u_j)\neq\emptyset \text{ a.e. on } \G_{i,j}\times(0,T)$,
\item for all $\psi\in \Dd(\overline{\O}\times [0,T))$,
\begin{equation}\label{gws_eq}
\begin{array}{c}
\ds\SiN\int_{\O_i}\int_0^T \phi_i u_i(x,t)\partial_t\psi(x,t)dxdt +\SiN\int_{\O_i}\phi_i u_0(x)\psi(x,0)dx\\
\ds-\SiN\int_{\O_i}\int_0^T \grad\phii(u_i(x,t))\cdot\grad\psi(x,t) dxdt=0.
\end{array}
\end{equation}

\end{enumerate}
\end{Def}

The third point of the previous definition, which insures the connection in the graph sense of the 
capillary pressures on the interfaces between several porous media, is well 
defined. Indeed, since $\phii(u_i)$ belongs to $L^2(0,T;H^1(\O_i))$, it admits 
a trace still denoted $\phii(u_i)$ on $\GijT$. Thanks to the 
remark~\ref{phii_invers}, we can define the trace of $u_i$ on $\GijT$.

\begin{rmk}\label{3bis}
One can equivalently substitute the condition:
$$3 bis.\ \ \ \breve{\pi}_i(u_i)\cap\breve{\pi}_j(u_j)\neq\emptyset \text{ a.e. on } \G_{i,j}\times(0,T),$$
to the third point of the definition~\ref{gws}, where $\breve{\pi}_i$ is the 
monotonous graph given by:
\begin{equation}\label{bpii}
\breve{\pi}_i(s)=\left\{\begin{array}{ll}
\ds \pi_i(s)  &\text{ if } s\in]0,1[,\\
\ds [\min_{j}(\a_j),\a_i]& \text{ if } s=0,\\
\ds [\b_i,\max_j(\b_j)] &\text{ if } s=1.
\end{array}\right.
\end{equation}
\end{rmk}

We will now quickly show the equivalence between the notion of weak 
solution to the problem~(\ref{P}) and the notion of weak solution given 
in \cite{EEM06}, in the 
case where this one is well defined, i.e. $N=2$ and 
$\max(\a_1,\a_2)=\a<\b=\min(\b_1,\b_2)$. We denote as in \cite{EEM06} the 
truncated capillary pressures by $\h\pi_1= \max(\a,\pi_1)$, 
$\h\pi_2=\min(\b,\pi_2)$, and we introduce the problem~(\ref{P2}), which is
treated in \cite{EEM06}. 
\begin{equation}\label{P2}\tag{$\widetilde{\mathcal{P}}$}
\left\{\begin{array}{ll}
\phi_i\partial_t u_i -\div(\l_i(u_i)\grad\pi_i(u_i))=0 &\text{in }\OiT,\\
\h\pi_1(u_1)=\h\pi_2(u_2)& \text{on }\GijT,\\
\l_1(u_1)\grad(\pi_1(u_1))\cdot {\bf n}_1+\l_2(u_2)\grad(\pi_2(u_2))\cdot 
{\bf n}_2=0 & \text{on }\GijT,\\
\l_i(u_i)\grad(\pi_i(u_i))\cdot {\bf n}_i=0 &\text{on }\partial\O_i\cap\partial\OT,\\
u_i(\cdot,0)=u_0(x) &\text{in }\O_i.
\end{array}\right.
\end{equation}
%==== inserer courbe des tronques ========================
\begin{figure}[htb]
\centering
\includegraphics[scale=.4]{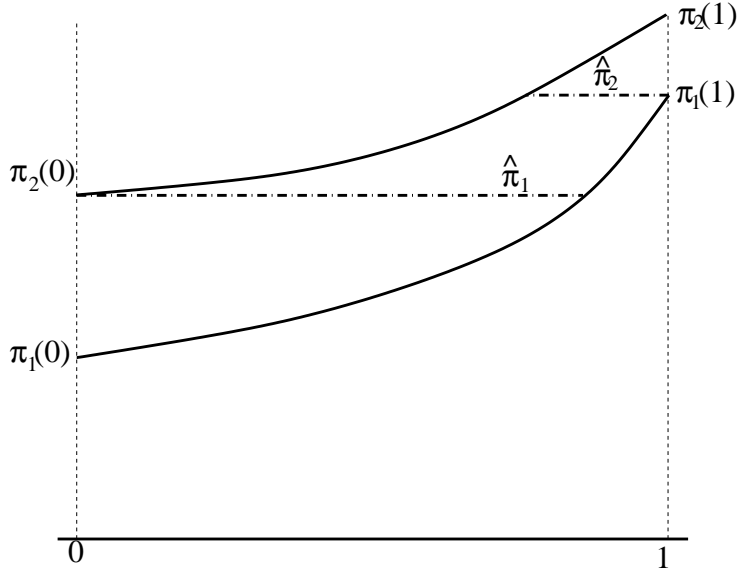}
\caption{Truncated capillary pressures}
\label{pression_tronques_nb} 
\end{figure}

Then it is easy to check
that: $\forall (s_1,s_2)\in[0,1]^2$,
\begin{equation}\label{graph&tronques}
\h\pi_1(s_1)=\h\pi_2(s_2) \LeftrightarrowÊ\t\pi_1(s_1)\cap\t\pi_2(s_2)\neq\emptyset
\LeftrightarrowÊ\breve{\pi}_1(s_1)\cap\breve{\pi}_2(s_2)\neq\emptyset.
\end{equation}

In order to recall the definition of weak solution, we have to introduce the function 
$$\Psi:\left\{\begin{array}{l}
[\a,\b]\rightarrow\R\\
\ds p\mapsto \int_\a^p \min_{j=1,2}(\l_j\circ\pi_j^{-1}(a))\textrm{d}a.
\end{array}\right.$$
$\Psi$ is increasing, and for $i=1,2$, $\Psi\circ\h\pi_i\circ\phii^{-1}$ is a Lipschitz continuous function.
\begin{Def}[weak solution to the problem~(\ref{P2})]\label{weak_sol_graph}
A function $u$ is said to be a weak solution to the problem~(\ref{P2}) 
if it verifies:
\begin{enumerate}
\item $u\in L^\infty(\OT)), 0\le u\le 1 \text{ a.e. in }\OT$,
\item $\forall i\in\{1,2\}, \phii(u_i)\in L^2(0,T;H^1(\O_i))$,
\item $w:\OT\rightarrow\R$, defined for $(x,t)\in \OiT$ by $w(x,t)=\Psi\circ\h\pi_i(u_i)(x,t)$ belongs to $L^2(0,T;H^1(\O))$,
\item for all $\psi\in \Dd(\overline{\O}\times [0,T))$,
$$
\begin{array}{c}
\ds\SiN \int_{\O_i}\int_0^T \phi_i u_i(x,t)\partial_t\psi(x,t)dxdt +\SiN\int_{\O_i}\phi_i u_0(x)\psi(x,0)dx\\
\ds-\SiN\int_{\O_i}\int_0^T \grad\phii(u_i(x,t))\cdot\grad\psi(x,t) dxdt=0.
\end{array}
$$
\end{enumerate}
\end{Def}
\begin{rmk}
The notion of weak solution to the problem~(\ref{P2}) can be adapted in the case where there are $N>2$ 
homogeneous domains, but we have to keep conditions of compatibility on 
$(\a_i)_{1\le i\le N} $ and  $(\b_i)_{1\le i\le N}$. 
\end{rmk}
{\bf Proof of the equivalence of the weak solutions}

On the one hand, if $u$ is a weak solution to the problem~(\ref{P2}) 
in the sense of definition \ref{weak_sol_graph}, then for a.e. $t\in(0,T)$, 
$w(\cdot,t)\in H^1(\O)$,and particularly $w(\cdot,t)$ admits a trace on $\Gij$, 
whose value is in the same time $\Psi(\h\pi_i(u_i(\cdot,t)))$ and 
$\Psi(\h\pi_j(u_j(\cdot,t)))$. Since $\Psi$ in increasing, for a.e $(x,t)\in\GijT$,  
$\h\pi_i(u_i(x,t))=\h\pi_j(u_j(x,t))$. Using~(\ref{graph&tronques}), 
we conclude that any weak solution to the problem~(\ref{P2}) is a weak solution 
to the problem~(\ref{P}) in the sense of definition~\ref{gws}.

On the other hand, if $u$ is a weak solution to the problem~(\ref{P}) 
in the sense of definition~\ref{gws}, then thanks to~(\ref{graph&tronques}), 
for almost every $(x,t)\in\GijT$,
\begin{equation}\label{graph_ench}
\h\pi_i(u_i(x,t))=\h\pi_j(u_j(x,t))   \Leftrightarrow  \Psi\circ\h\pi_i\circ\phii^{-1}(\phii(u_i(x,t)))=\Psi\circ\h\pi_j\circ\phij^{-1}(\phij(u_j(x,t))).
\end{equation}
Since $\Psi\circ\h\pi_i\circ\phii^{-1}$ is a Lipschitz continuous function, 
the second point in definition~\ref{gws} insures us that $\Psi\circ\h\pi_i(u_i)$ 
belongs to $L^2(0,T,H^1(\O_i))$ for $i=1,2$, and 
(\ref{graph_ench}) insures the connection of the traces on $\GijT$, 
then the third point of definition~\ref{weak_sol_graph} is fulfilled and $u$ is a 
weak solution to the problem~(\ref{P2}). 
\CQFD

\begin{rmk}
We can define a {\em function} $\t\pi^{-1}_i, i\in\IN$, which verifies  
$\t\pi^{-1}_i\circ\t\pi_i(s)=s$ for any $s\in [0,1]$. Using the function 
defined on $\R$ by $\t\Psi(p)=\int_{-\infty}^p \min_{j=1,2}(\l_j\circ\t\pi_j^{-1}(a))
\textrm{d}a$, it is easy to check that  we can equivalently substitute the function 
$\t\Psi\circ\pi_i(u_i)$ to $\Psi\circ\h\pi_i(u_i)$ in the third point of 
definition~\ref{weak_sol_graph}. This function is still defined if $\a\ge\b$, but it 
becomes identically $0$, so the notion of weak solution to the problem~(\ref{P2}) 
is weaker than the notion of weak solution to the problem~(\ref{P}). Indeed, in such a 
case, $u(x,t)=u_0(x)=a\in]0,1[$ for any 
$(x,t)\in\OT$ is a weak solution to the problem~(\ref{P2}), 
but it does not fulfill the third point in 
definition~\ref{gws}.
\end{rmk}

% 3) EXISTENCE D'UNE SOLUTION GRAPHE
\section{Existence of a weak solution}\label{existence}
The aim of this section is to prove the following theorem, which claims the existence of a 
 weak solution to the problem~(\ref{P}). This result has already been proven in 
 section~\ref{notion} in the case $N=2$ and $\a>\b$, for which the notion of weak 
 solution in the sense of definition~\ref{gws} is equivalent to the notion of weak 
 solution in the sense of definition~\ref{weak_sol_graph}.
\begin{thm}[Existence of a weak solution]\label{thm_existence}
Under assumptions~\ref{geom} and \ref{pii_li}, there exists a  weak solution to 
problem~(\ref{P}) in the sense of definition~\ref{gws}.
\end{thm}

\begin{proof}
In order to prove the existence of a weak solution to the problem~(\ref{P})
 in the sense of the definition~\ref{gws}, we build a sequence of solutions to 
 approximated problems~(\ref{Pn}), which converges, up to a subsequence, toward 
 a weak solution to the problem~(\ref{P}). The approximated problems do not involve 
 capillary barriers, so existence and uniqueness of  such approximated 
 solutions is given in~\cite{NoDEA}.
 We let the proof of the following technical lemma to the reader.
\begin{lem}\label{phiin}
There exists sequences $(\l_{i,n})_n$, $(\pi_{i,n})_n$ belonging to $(C^\infty([0,1],\R))^{\N}$
 such that, for $i\in\IN$, and for $n$ large enough:
\begin{itemize}
\item $\ds {\l_{i,n}}_{| [0,1/n]\cup[1-1/n,1]}=\frac{1}{n^2}$,
 $\ds \l_{i,n}(s)> \frac{1}{2n^2},$ for all $s\in[0,1]$, $\l_{i,n}\rightarrow \l_i$ uniformly on $[0,1]$,
\item $\pi_{i,n}(0)=\pi_{j,n}(0) \rightarrow -\infty $, $\pi_{i,n}(1)=\pi_{j,n}(1) \rightarrow +\infty $, $Kn^{\frac{3}{2}}>\pi_{i,n}' \ge \ds\frac{1}{n}$,
 $\pi_{i,n}\rightarrow \pi_i$ in $L^1(0,1)$,
 $\pi_{i,n}\rightarrow \pi_i$ and $\pi_{i,n}'\rightarrow \pi_i'$ uniformly on any compact set of $]0,1[$,
\item the function $\phiin\!: s\mapsto\!\! \int_0^s \l_{i,n}(a)\pi_{i,n}'(a)da$ furthermore fulfills
 $\phiin([0,1])=\phii([0,1])$ and $\phiin \rightarrow \phii$ in  $W^{1,\infty}(0,1)$.
\end{itemize}
\end{lem}

 We also define the increasing functions:
 $$
\Psi_n: \left\{\begin{array}{l}
 [a_n,b_n]\rightarrow\R\\
\ds p\mapsto \int_{a_n}^p \min_{j\in\IN}(\l_{j,n}\circ\pi_{j,n}^{-1}(a))
\textrm{d}a.
\end{array}\right.
 $$ 
 The conditions on the functions on the intervals
 $[0,\frac{1}{n}]\cup[1-\frac{1}{n},1]$ insures that for any fixed large $n$, the
 functions $(\phiin\circ\pi_{i,n}^{-1}\circ\Psi_n^{-1})'$ are Lipschitz
 continuous.  Then thanks to \cite{NoDEA}, for all $n$,  the approximated problems:
 \begin{equation}\label{Pn}
\left\{\begin{array}{ll}
\phi_i \partial_t u_{i,n} -\div(\l_{i,n}(u_{i,n})\grad\pi_{i,n}(u_{i,n}))=0 
&\text{in }\OiT,\\
\pi_{i,n}(u_{i,n})=\pi_{j,n}(u_{j,n}) & \text{on }\GijT,\\
\l_{i,n}(u_{i,n})\grad(\pi_{i,n}(u_{i,n}))\!\cdot\! {\bf n}_i\!+\!\l_{j,n}(u_{j,n})
\grad(\pi_{j,n}(u_{j,n}))\!\cdot\! {\bf n}_j=0 & \text{on }\GijT,\\
\l_{i,n}(u_{i,n})\grad(\pi_{i,n}(u_{i,n}))\cdot {\bf n}_i=0 &\text{on }
\partial\O_i\cap\partial\OT,\\
u_{i,n}(x,0)=u_0(x) &\text{in }\O.
\end{array}\right.
\end{equation}
 admit a unique weak solution in the sense of definition~\ref{weak_sol_n} given
 below, and this solution belongs to $C([0,T],L^p(\O))$ for $1\le p<+\infty$.
 \begin{Def}\label{weak_sol_n}
{\bf (Weak solutions for approximated problems)}\\ 
A function $u_n$ is said to be a weak solution to the problem~(\ref{Pn}) 
if it verifies:
\begin{enumerate}
\item $u_n\in L^\infty(\OT)), 0\le u_n\le 1 \text{ a.e. in }\OT$,
\item $\forall i\in\{1,2\}, \phiin(u_{i,n})\in L^2(0,T;H^1(\O_i))$,
\item $w_n:\OT\rightarrow\R$, defined on $\OiT$ by $w_n=\Psi_n\circ\pi_{i,n}(u_{i,n})$ belongs to \\
$L^2(0,T;H^1(\O))$,
\item for all $\psi\in \Dd(\O\times [0,T))$,
\begin{equation}\label{eq_n}
\begin{array}{c}
\ds\SiN \int_{\O_i}\int_0^T \phi_i u_{i,n}(x,t)\partial_t\psi(x,t)dxdt +\SiN\int_{\O_i} \phi_i u_0(x)\psi(x,0)dx\\
\ds-\SiN\int_{\O_i}\int_0^T \grad\phiin(u_{i,n}(x,t))\cdot\grad\psi(x,t) dxdt=0.
\end{array}
\end{equation}
\end{enumerate}
\end{Def}
 The proof of existence of a weak solution given in \cite{NoDEA}, shows that for all $i\in\IN$, for all $n$, there exists $C_1>0$ not depending on $n$ such that, for all $i\in\IN$:
 \begin{equation}\label{L2H1_estimate}
 \| \phiin(u_{i,n})\|^2_{L^2(0,T;H^1(\O_i))} \le C_1\|\pi_{i,n}\|_{L^1(0,1)},
 \end{equation}
thus $( \phiin(u_{i,n}))_n$ is a bounded sequence of $L^2(0,T;H^1(\O_i))$ using lemma~\ref{phiin}. A study of the proof of the time translate estimate used in \cite{NoDEA,EEM06}, and detailed in 
\cite[lemma 4.6]{EGH00} leads to the existence of $C_2$ not depending on $n$ such that: 
\begin{equation}\label{time_trans}
\| \phiin(u_{i,n}(\cdot,\cdot+\tau))-\phiin(u_{i,n}(\cdot,\cdot))\|^2_{L^2(\O_i\times(0,T-\tau))}\le \tau C_2
\|\pi_{i,n}\|_{L^1(0,1)}\|\varphi_{i,n}'\|_{L^\infty(0,1)}.
\end{equation}
 Using lemma~\ref{phiin} once again, estimates~(\ref{L2H1_estimate}), (\ref{time_trans}) allow us to apply  Kolmogorov's compactness criterion (see e.g. \cite{Bre83}), thus we can claim the relative compactness of the sequence $(\phiin(u_{i,n}))_n$ in $L^2(\OiT)$. 
 There exists $f_i\in L^2(0,T;H^1(\O_i))$ such that
 $$\phiin(u_{i,n}) \rightarrow f_i \text{  in } L^2(\OiT),$$
 $$\phiin(u_{i,n}) \rightarrow f_i \text{ weakly in }L^2(0,T;H^1(\O_i)).$$

Let us now recall a very useful lemma, classically called Minty trick, and introduced in this framework by Leray and Lions in the famous paper~\cite{LL65}.
\begin{lem}[Minty trick]
Let $(\phi_n)_n$ be a sequence of non-decreasing functions with for all $n$, $\phi_n:\R\rightarrow\R$, and let $\phi:\R\rightarrow \R$ be a non-decreasing continuous function such that:
\begin{itemize}
\item $\phi_n\rightarrow \phi$ pointwise,
\item there exists $g\in L^1_{loc}(\R)$ such that $|\phi_n|\le g$.
\end{itemize}
Let $\Oo$ be an open subset of $\R^k$, $k\ge 1$. Let $(u_n)_n\in (L^\infty(\Oo))^\N$, let $u\in L^\infty(\Oo)$ and let $f\in L^1(\Oo)$ such that:
\begin{itemize}
\item $u_n\rightarrow u$ in the $L^\infty(\Oo)$-weak-$\star$ sense,
\item $\phi_n(u_n)\rightarrow f$ in $L^1(\Oo)$.
 \end{itemize}
 Then 
 $$ f=\phi(u).$$
\end{lem}

 Since $0\le u_{i,n} \le 1$, $(u_{i,n})_n$ converges up to a subsequence to $u_i$ in the $L^\infty(\OiT)$-weak-$\star$ sense. $(\phiin)_n$ converges uniformly toward $\phii$ on $[0,1]$, and we can easily check, using Minty trick, 
 that $f_i=\phii(u_i)\in L^2(0,T;H^1(\O_i))$. Thus we can pass to the limit in the formulation~(\ref{eq_n})
 to obtain the wanted weak formulation:
 $$
 \begin{array}{c}
\ds\SiN\int_{\O_i}\int_0^T \phi_i u_i(x,t)\partial_t\psi(x,t)dxdt +\SiN \int_{\O_i} \phi_i u_0(x)\psi(x,0)dx\\
\ds-\SiN\int_{\O_i}\int_0^T \grad\phii(u_i(x,t))\cdot\grad\psi(x,t) dxdt=0.
\end{array}
 $$
 
 The last point needed to achieve the proof of theorem~\ref{thm_existence} is the convergence of the traces of the approximate solutions $(u_{i,n})_n$ on $\GijT$ toward the trace of $u_i$, and to verify that $\t\pi_i(u_i)\cap\t\pi_j(u_j)\neq\emptyset$ a.e. on $\GijT$.
 
Since $\O_i$ has a Lipschitz boundary, there exists an operator $P$, continuous from 
$H^1(\O_i)$ into $ H^1(\R^d)$, and also from $L^2(\O_i)$ into $L^2(\R^d)$, such that $Pv_{|\O_i}=v$ for all $v\in L^2(\O_i)$. Then $P$ is continuous from $H^s(\O_i)$ into $H^s(\R^d)$ for all $s\in[0,1]$.
One has, for all $v\in H^s(\O_i)$, 
$$\| v \|_{H^s(\O_i)}\le \| Pv \|_{H^s(\R^d)}\le\|Pv\|_{H^1(\R^d)}^s \|Pv\|_{L^2(\R^d)}^{1-s}\le 
C \|v \|_{H^1(\O_i)}^s \|v\|_{L^2(\O_i)}^{1-s}.$$
One deduces from the previous inequality and from (\ref{time_trans}) that for all $s\in ]0,1[$, for all $\tau \in ]0,T[$, there exists $C_3$ not depending on $n, \tau$ such that
\begin{equation}\label{times_trans2}
\|  \phiin(u_{i,n}(\cdot,\cdot+\tau))-\phiin(u_{i,n}(\cdot,\cdot))\|^2_{L^2(0,T-\tau;H^s(\O_i))}\le\tau^{1-s} C_3 
\end{equation}
For $s_1>s_2$, $H^{s_1}$ is compactly imbedded in $H^{s_2}$, and then estimate~(\ref{times_trans2}) allows us 
to claim that the sequence 
$(\phiin(u_{i,n}))_n$ is relatively compact in $L^2(0,T;H^s(\O_i))$ for all 
$s\in]0,1[$. Particularly, one can extract a subsequence converging toward  
$\phii(u_i)$ in $L^2(0,T;H^s(\O_i))$.
We can claim, using once again Minty trick, that the traces of $(\phiin(u_{i,n}))_n$ on $\Gij$ also converge toward the trace of $\phii(u_i)$, still denoted $\phii(u_i)$ in $L^2(0,T;H^{s-1/2}(\Gij))$, and particularly for almost every $(x,t)\in\GijT$. Since $\phii$ is increasing, $(u_{i,n}(x,t))_n$ converges almost everywhere on $\GijT$ toward $u_i(x,t)$.

Let us now check that $\t\pi_i(u_i)\cap\t\pi_j(u_j)\neq\emptyset$ a.e. on $\GijT$. For almost every $(x,t)\in\GijT$ the sequence 
$(\pi_{i,n}(u_{i,n}(x,t)))_n$ converges (up to a new extraction) toward $\gamma_i(x,t)\in\overline\R$. Since for all $n$, 
$\pi_{i,n}(u_{i,n}(x,t))=\pi_{j,n}(u_{j,n}(x,t))$, one has: 
\begin{equation}\label{traces_lim}
\gamma_i(x,t)=\gamma_j(x,t) \text{ a.e. on }\GijT.
\end{equation}

If $u_i(x,t)\in\ ]0,1[$, then $\gamma_i(x,t)=\pi_i(u_i(x,t))$. If $u_i(x,t)=0$, $\gamma_i(x,t)\le \a_i$, and $\gamma_i(x,t)\in\t\pi_i(0)$. In the same way, if $u_i(x,t)=1$, $\gamma_i(x,t)\in\t\pi_i(1)$. 

This achieves the proof of theorem \ref{thm_existence}, because relation~(\ref{traces_lim}) insures the connection of the traces in the sense of:
$$
\t\pi_i(u_i)\cap\t\pi_j(u_j)\neq\emptyset  \text{ a.e. on }\GijT.
$$
\end{proof}

% A REGULARITY RESULT

\section{A regularity result}\label{section_regularity}
\par In this section and in section \ref{unicite}, we show the existence and 
the uniqueness of a  solution with bounded flux to the problem~(\ref{P}) in the 
one-dimensional case. We make the proofs in the case where there are only two 
sub-domains $\O_1=]-1,0[$ and $\O_2=]0,1[$, but a straightforward adaptation 
of them gives the same result  for 
an arbitrary finite number of $\O_i$, each one with an arbitrary finite measure.
We now state the main result of this section, which claims the existence of a
solution with  bounded spatial derivatives on $\Qq_i$, where $\Qq_i=$ 
$\OiT$. We also set $\Qq=]-1,1[\times]0,T[$ and $\G=\{x=0\}$.
\begin{thm}[Existence of a bounded flux solution]\label{regularity_graph}
Let $u_0 \in L^\infty(-1,1)$, $0\le u_0\le 1$ such that:
\begin{itemize}
\item $\phii(u_0) \in W^{1,\infty}(\O_i)$,
\item $\t\pi_1(u_{0,1})\cap \t\pi_2(u_{0,2})\neq\emptyset$ on $\G$.
\end{itemize}
Then there exists a weak solution $u$  to the 
problem~(\ref{P}) such that $\partial_x\phii(u_i)\in L^\infty(\Qq_i)$.
\end{thm}

All the section will be devoted to the proof of the theorem~\ref{regularity_graph}. 
As in section~\ref{existence}, we will get this existence result by taking 
the limit of a sequence of solutions to approximate problems~(\ref{Pn}) involving no 
capillary barriers, whose data fulfill the properties stated in lemma~\ref{phiin}.

\begin{proof}
We will now build a sequence of approximate initial data $(u_{0,n})$ adapted to the sequence of approximate problems.
\begin{lem}\label{reg_u0}
Let $u_0$ be chosen as in theorem~\ref{regularity_graph}, then there exists $(u_{0,n})_n$ such that, for all $n$,
\begin{itemize}
\item $0\le u_{0,n} \le 1$,
\item $\pi_{1,n}(u_{0,n,1})=\pi_{2,n}(u_{0,n,2})$ on $\G$.
\end{itemize}
The sequence $(u_{0,n})_n$ furthermore fulfills:
\begin{equation}\label{conv_u0n}
 \lim_{n\rightarrow\infty} \| u_{0,n}-u_0\|_\infty= 0, \qquad \|\partial_x \phiin(u_{0,n})\|_{L^\infty(\O_i)} 
 \le \|\partial_x \phii(u_{0})\|_{L^\infty(\O_i)} .
 \end{equation}
\end{lem}
\begin{proof}
Since $\t\pi_1(u_{0,1})\cap\t\pi_2(u_{0,2})\neq \emptyset $, then there exists $(a_{1,n}, a_{2,n})\in [0,1]^2$ 
such that one has $\pi_{1,n}(a_{1,n})=\pi_{2,n}(a_{2,n})$ and $|a_{1,n}-u_{0,1}|+ |a_{2,n}-u_{0,2}|\rightarrow 0$.
One sets, for $x\in\O_i$: 
$$u_{0,n}(x)=\phiin^{-1}\left(T_{\phii}\left[\phii(u_0)+\phiin(a_{i,n})-\phii(u_{0,i})\right]\right)
$$
where 
$$T_{\phii}(s)=\left\{\begin{array}{rcl}
s&\text{ if }& s\in[0,\phii(1)]=[0,\phiin(1)],\\
\phiin(1)&\text{ if }& s>\phii(1),\\
0&\text{ if }& s<0.
\end{array}\right.$$
Then the sequence $(u_{0,n})$ converges uniformly toward $u_0$. For all $n$, $0\le u_{0,n}\le1$ and 
 either $\partial_x\phiin(u_{0,n})=\partial_x \phii(u_0)$, or $\partial_x\phiin(u_{0,n})=0$. 
 \end{proof}\\[10pt]
 
The approximate problem (\ref{Pn}) admits a unique solution  $u_n$ thanks 
to~\cite{NoDEA}, which belongs to $C([0,T],L^1(\O))$. Now, in order to get a  $L^\infty(\Qq_i)$-estimate on 
the sequence $(\partial_x\phiin(u_n))_n$, we  
 introduce a new family of approximate problems~(\ref{Pnk}) for 
 which the spatial dependence of the data is smooth.

Let $\theta \in C^\infty(\R), 0\le \theta\le 1 $, with $\theta(x)= 0$ if $x<-1$, and $\theta(x)=1$ if $x>1$. 
Let $k\in\N^\star$, one sets:
\begin{itemize}
\item $\phi^{k}(x)=(1-\theta(kx))\phi_1+\theta(kx)\phi_2$,
\item $ \l_{n,k}(s,x)=(1-\theta(kx))\lambda_{1,n}(s)+\theta(kx)\lambda_{2,n}(s),$
\item $ \pi_{n,k}(s,x)=(1-\theta(kx))\pi_{1,n}(s)+\theta(kx)\pi_{2,n}(s).$
\end{itemize}

We will now take a new approximation of the initial data.
$$u_{0,n,k}(x)=\left\{
\begin{array}{rcl}
u_{0,n}\left(\frac{k}{k-1}\left(x+\frac{1}{k}\right)\right) & \text{ if } & x<-1/k,\\
u_{0,n}\left(\frac{k}{k-1}\left(x-\frac{1}{k}\right)\right) & \text{ if } & x>1/k.\\
\end{array}\right.$$
In the layer $[-1/k,1/k]$, $u_{0,n,k}$ is defined by the relation
$$(1-\theta(kx))\pi_{1,n}(u_{0,n,k}(x))+\theta(kx)\pi_{2,n}(u_{0,n,k}(x))=\pi_{1,n}(a_{1,n})
=\pi_{2,n}(a_{2,n}),$$
so that the approximate capillary pressure $\pi_{n,k}(u_{0,n,k},\cdot)$ is constant through the layer.

Moreover one has  either $$ \l_{n,k}(u_{0,n,k},x)\partial_x (\pi_{n,k}(u_{0,n,k},x))=
\frac{k}{k-1}\partial_x\phiin(u_{0,n})\qquad \textrm{ if }|x|>\frac{1}{k},$$ or 
$$\partial_x 
(\pi_{n,k}(u_{0,n,k},x))=0 \qquad
\textrm{ if }|x|<\frac{1}{k}.$$
 So we directly deduce from the definition of $u_{0,n,k}$ 
the following lemma:
\begin{lem}\label{u_0nk}
Let $n\ge1$, $0\le u_{0,n}\le 1$ with $\phiin(u_{0,n}) \in W^{1,\infty}(\O_i)$ 
and $\pi_{1,n}(u_{0,n,1})=\pi_{2,n}(u_{0,n,2})$, then there exists a sequence 
$(u_{0,n,k})_k$ satisfying, for all $k\ge2$, that $0\le u_{0,n,k} \le 1$ and 
$$\| \l_{n,k}(u_{0,n,k},\cdot)\partial_x (\pi_{n,k}(u_{0,n,k},\cdot)) 
\|_\infty \le 
2\max_{i=1,2}(\|\partial_x\phiin(u_{0,n})\|_\infty), $$
$$u_{0,n,k}\rightarrow u_{0,n}\text{ in } L^1(\O) \text{ as } k\rightarrow +
\infty.
$$
\end{lem}
For any fixed $k\ge2$ and $n$ large enough, we can 
now introduce the smooth non-degenerate parabolic problem (\ref{Pnk}):
\begin{equation}\label{Pnk}
\left\{\begin{array}{l}
\phi^{k}(x)\partial_t u_{n,k}-\partial_x (\l_{n,k}( u_{n,k},x)\partial_x \pi_{n,k}( u_{n,k},x))=0,   \\
\partial_x u_{n,k}(-1,t)=\partial_x u_{n,k}(1,t)=0,\\
 u_{n,k}(x,0)=u_{0,n,k}(x).
\end{array}\right.\end{equation}
Moreover, one  can furthermore suppose, up to a new regularization, that $u_{0,n,k}\in C^\infty([-1,1])$.
Then  (\ref{Pnk}) admits a unique strong solution $u_{n,k}\in 
C^\infty([0,T]\times [-1,1])$ (see for instance \cite{F64,LSU}).

Now one sets $f_{n,k}(x,t)=\l_{n,k}( u_{n,k},x)\partial_x \pi_{n,k}( u_{n,k},x)$, so the main equation of 
(\ref{Pnk}) can be rewritten:
$$
\phi^{k}\partial_t u_{n,k}=\partial_x f_{n,k}.
$$
A short calculation shows that $f_{n,k}(x,t)$ is the solution of the problem:
\begin{equation}\label{Pb_flux}
\left\{\begin{array}{l}
\partial_t f_{n,k}= a_{n,k} \partial_{xx}^2 f_{n,k}+b_{n,k} \partial_{x} f_{n,k},\\
f_{n,k}(-1,t)=f_{n,k}(1,t)=0,\\
f_{n,k}(x,0)=\l_{n,k}(u_{0,n,k},\cdot)\partial_x (\pi_{n,k}(u_{0,n,k},\cdot)),
\end{array}\right.
\end{equation}
where $a_{n,k},b_{n,k}$ are the regular functions defined below.
$$a_{n,k}=\l_{n,k}(u_{n,k},x)\frac{(\pi_{n,k})'(u_{n,k},x)}{\phi^{k}(x)}>0,$$
$$b_{n,k}=(\l_{n,k})'(u_{n,k},x)\frac{\partial_x [\pi_{n,k}(u_{n,k},x)]}{\phi^{k}(x)}
+ \l_{n,k}(u_{n,k},x)\partial_x\left[\frac{(\pi_{n,k})'(u_{n,k},x)}{\phi^{k}(x)}\right] .
$$
The fact that $u_{0,n,k}$ is supposed to be regular allows us to write the problem~(\ref{Pb_flux}) in a strong  
sense (this is necessary, because this problem can not be written in a conservative form).
In particular,
$f_{n,k}$ satisfies  the maximum principle, and thus
$$\| f_{n,k} \|_{L^\infty((-1,1)\times(0,T))} \le \| \l_{n,k}(u_{0,n,k},\cdot)\partial_x (\pi_{n,k}(u_{0,n,k},\cdot))
 \|_{L^\infty(-1,1)}.
$$
Thanks to the lemmas \ref{u_0nk} and \ref{reg_u0}, we have a uniform bound on $(f_{n,k})$:
\begin{equation}\label{maxPnk}
\| f_{n,k} \|_{L^\infty((-1,1)\times(0,T))} \le 2 \max_{i=1,2}(\|\partial_x\phii(u_0)\|_\infty).
\end{equation}

Since the problem~(\ref{Pnk}) is fully non degenerated (recall that $
\l_{i,n}> \frac{1}{2n^2}$  and $\pi_{i,n}' \ge \ds\frac{1}{n}$)
it follows that  $\partial_x u_{n,k}$ and $\partial_t u_{n,k}$ are uniformly bounded respectively in $L^\infty(\Qq_i)$ 
and in $L^2(0,T:H^{-1}(\O_i))$ with respect to  $k$, then the sequence $(u_{n,k})_k$ converges toward $u_n$ in 
$L^2(\Qq_i)$, and the limit $u_n$ fulfills, thank to estimate~(\ref{maxPnk}):
\begin{equation}\label{maxPn}
\| \partial_x \phiin(u_n)\|_{L^\infty(\Qq_i)} \le 2 \max_{i=1,2}(\|\partial_x
\phii(u_0)\|_\infty).
\end{equation}

One has for all $\psi\in\Dd([-1,1]\times [0,T[)$,
\begin{equation}\label{weak_kn}
\int_0^T \int_{-1}^1\phi^{k} u_{n,k}\partial_t \psi+  \int_{-1}^1\phi^{k} 
u_{0,n}^{k} \psi_0 -
\int_0^T \int_{-1}^1 f_{n,k}\partial_x\psi=0.
\end{equation}
Thanks to (\ref{maxPnk}), $$\lim_{k\rightarrow+\infty}\int_0^T \int_{-\frac{1}
{k}}^{\frac{1}{k}} f_{n,k}\partial_x\psi=0.$$
One has $u_{n,k}\rightarrow u_n$ in the $L^\infty(\Qq)$-weak-
$\star$ and $L^2(\Qq)$ senses, $u_{0,n,k}\rightarrow u_{0,n}$ in $L^1(-1,1)$ thanks to 
lemma~\ref{u_0nk}. Moreover, thanks to estimate~(\ref{maxPnk}), 
$\partial_x\pi_{i,n,k}(u_{n,k})\rightarrow \partial_x\pi_{i,n}(u_{n,k})$ in the $L^\infty(\Qq)$-weak-
$\star$ sense.
Thus we can let $k$ tend toward $+\infty$ in~(\ref{weak_kn}) to get
\begin{equation}\label{weak_n}
\int_0^T \sum_{i=1,2}\int_{\O_i} \phi_i u_n\partial_t \psi+ 
\sum_{i=1,2}\int_{\O_i}\phi_i u_{0,n}\psi_0 -
\int_0^T \sum_{i=1,2}\int_{\O_i} \l_{i,n}(u_n)\partial_x \pi_{i,n}(u_n)
\partial_x\psi=0.
\end{equation}
Furthermore, using the fact that   $\pi_{n,k}(u_{n,k},x)$ belongs to $L^2(0,T; H^1(\O))$ and, even more, that  
$\partial_x(  \pi_{n,k}(u_{n,k},x) ) $ is   bounded uniformly in $k$,  we can claim that 
$\pi_{1,n}(u_{1,n})=\pi_{2,n}(u_{2,n})$, and so $u_n$ is the unique weak 
solution to the approximate problem~(\ref{Pn}) for $u_{0,n}$ as initial data.

When $n$ tends toward $+\infty$, the sequence $(u_n)_n$ converges, up to a 
subsequence toward a weak solution to the problem~(\ref{P}), as seen in section~\ref{existence}, 
but the estimate~(\ref{maxPn})  insures that  
$$ \partial_x \phii(u)\in
L^\infty(\Qq_i).$$ 
This achieves the proof of theorem~\ref{regularity_graph}.
\end{proof}

% UNIQUENESS
\section{A uniqueness result}\label{unicite}

In this section, we give a uniqueness result in the one dimensional case 
in a
framework where the existence results 
are stronger than the general existence result stated in
theorem~\ref{thm_existence}. 
Under a regularity assumption on the initial data $u_0$, we proved in
section~\ref{section_regularity} the existence of a  solution having 
bounded flux, for which we give a uniqueness result in 
theorem~\ref{prop_unicite} and corollary~\ref{cor_uni}. The bound on the flux will be necessary to prove that the contraction
property is also available in the neighborhood of the interface $\{x=0\}$.
Then we show in
theorem~\ref{existence_SOLA} the existence and uniqueness of the weak
solution which is the limit of  bounded flux solutions for any initial data $u_0$
with $0\le u_0\le 1$. 
Indeed, the set of initial data giving a bounded flux solution is dense in $L^\infty(\O)$ for the $L^1(\O)$ topology, 
and theorem~\ref{prop_unicite} has for consequence that the contraction property can be extended to a larger class of 
solution, defined for all initial data in $L^\infty(\O)$.
We unfortunately are not
able to characterize them differently than by  a limit of bounded flux solutions, and we can not either exhibit a weak
solution which is not the limit of bounded flux solutions.

\begin{thm}[$L^1$-contraction principle for bounded flux solutions]\label{prop_unicite}
Let $u,v$  be two weak solutions to the problem~(\ref{P}) for the initial 
data $u_0, v_0$. Then, if $\partial_x\phii(u_i)$ and $\partial_x\phii(v_i)$ belong to 
$L^\infty(\Qq_i)$, we have the following 
$L^1$-contraction principle: $\forall t\in [0,T]$,
\begin{equation}\label{contract}
\sum_{i=1,2}\int_{\O_i} \phi_i \left(u(x,t)-v(x,t)\right)^\pm dx \le \sum_{i=1,2} 
\int_{\O_i} \phi_i \left(u_{0}(x)-v_{0}(x)\right)^\pm dx.
\end{equation}
\end{thm}
The first part of this section is devoted to the proof of the theorem \ref{prop_unicite} which, 
 with theorem~\ref{regularity_graph}, admits the following straightforward 
consequence:
\begin{coro}[Uniqueness of the bounded flux solution]\label{cor_uni}
For all  $u_0\!\in \!L^{\infty}(-1,1)$ with $0\le u_0\le 1$, such that, for $i=1,2$, 
$\phii(u_0)\in W^{1,\infty}(\O_i)$, and   $\t\pi_1(u_{0,1})\cap\t
\pi_2(u_{0,2})\neq\emptyset$, there exists a unique weak solution 
to the problem~(\ref{P}) in the sense of definition~\ref{gws} and such that  $\partial_x\phii(u)\in L^\infty(\Qq_i)$;  
moreover $u\in C([0,T],L^p(\O))$ for all $1\le p < +\infty$.
\end{coro}

\begin{proof}
The proof of the theorem \ref{prop_unicite} is based on entropy inequalities, obtained through the  method of doubling variables, 
first introduced by S. Kru\u{z}kov~\cite{K70} for first order equations, and then 
adapted by  J. Carrillo~\cite{Car99} for degenerate parabolic problems.
Note that in the present setting, we only need doubling with respect to the time--variable, as it is done, for instance by F. Otto~ \cite{Otto96}  for elliptic--parabolic problems (or in  \cite{BP05} for Stefan--type problems).

In the sequel of the proof, we will only give the comparison
$$
\sum_{i=1,2}\int_{\O_i} \phi_i \left(u(x,t)-v(x,t)\right)^+ dx \le \sum_{i=1,2} 
\int_{\O_i} \phi_i \left(u_{0}(x)-v_{0}(x)\right)^+ dx.
$$
The comparison with $(\cdot)^-$ instead of $(\cdot)^+$ can be proven  exactly the
same way.

Let $u$ be a bounded flux solution to the one-dimensional problem, i.e 
$\partial_x\phii(u)\in L^\infty(\Qq_i)$, $i=1,2$. 
The weak formulation of definition~\ref{gws} adapted to the one-dimensional 
framework of the section can be rewritten, for all $\psi\in \Dd(\overline{\O} 
\times [0,T[)$,
\begin{equation}\label{weak_for_1d}
\begin{array}{c}
\ds\int_0^T\sum_{i=1,2}\int_{\O_i} \phi_i u(x,t)\partial_t \psi(x,t)dxdt 
+\sum_{i=1,2}\int_{\O_i} \phi_i u_0(x)\psi(x,0)dx\\
\ds- \int_0^T\sum_{i=1,2}\int_{\O_i} \partial_x\phii(u)(x,t)\partial_x
\psi(x,t)dxdt=0
\end{array}
\end{equation}
This formulation clearly implies, for $i=1,2$, for all $\psi\in C^\infty_c(\overline\O_i\times [0,T[)$ with $\psi(0,t)=0$, 
\begin{equation}\label{weak_i}\begin{array}{c}
\ds \int_0^T\int_{\O_i}  \phi_i u(x,t)\partial_t \psi(x,t)dxdt +
\int_{\O_i} \phi_i u_0(x)\psi(x,0)dx \\
\ds - \int_0^T\int_{\O_i} \partial_x\phii(u)(x,t)\partial_x
\psi(x,t)dxdt=0
\end{array}
\end{equation}
Classical computations (see e.g. \cite{BP05,Car99,Otto96}) on equation~(\ref{weak_i}) lead to the following entropy inequalities: 
for all  weak solutions $u,v$,  for initial data $u_0,v_0$, for all $\xi\in \Dd^+(\overline\O_i\times[0,T[\times[0,T[)$ such that $\xi(0,t,s)=0$,
\begin{equation}\label{entro_i}
\begin{array}{c}
\ds\int_0^T \int_0^T \int_{\O_i}\phi_i(u(x,t)-v(x,s))^+(\partial_t\xi(x,t,s)+\partial_s\xi(x,t,s))dxdtds\\
+\ds \int_0^T\!\int_{\O_i}\phi_i(u_0(x)-v(x,s))^+\xi(x,0,s)dxds \\
\ds + \int_0^T \int_{\O_i}\phi_i(u(x,t)-v_0(x))^+\xi(x,t,0)dxdt\\
-\ds \int_0^T \int_0^T \int_{\O_i}\partial_x(\phii(u)(x,t)-\phii(v)(x,s))^+\partial_x\xi(x,t,s)dxdtds\ge0.
\end{array}
\end{equation}

Let us note here an important consequence of the entropy inequality 
(\ref{entro_i}) (and of the 
corresponding one for $(u-v)^-$ ),  namely that 
$u$ can be proved to satisfy
\begin{equation}\label{ess-cont}
ess-\!\!\lim_{t\rightarrow 0}\,\, \int_{\O_i}|u(x,t)-u_0(x)|dx =0\,.
\end{equation}
Indeed, this follows by taking $v$ as a  constant in (\ref{entro_i}) and using   an approximation  argument, see e.g. Lemma 7.41 in~\cite{MNRR96}.  
We deduce the time continuity at $t=0$ for any  solution and in particular for both $u$ and $v$ taken above.

Now, let $\rho\in C^\infty_c(\R,\R^+)$ with $supp(\rho)\subset [-1,1]$ and 
$\int_{\R}\rho(t)dt=1$. One denotes $\rho_m(t)=m\rho(mt)$. 
Let $\psi\in \Dd^+([-1,1]\times[0,T[)$ with 
$\psi(0,\cdot)=0$. For $m$ large enough, $\xi(x,t,s)= \psi(x,t)\rho_m(t-s)$ belongs to $\Dd^+([-1,1]\times[0,T[\times[0,T[)$, and we can take it as test 
function in (\ref{entro_i}). Then summing on $i=1,2$ leads to
\begin{equation}\label{avant_cont}
\begin{array}{c}
\ds \int_0^T\int_0^T\sum_{i=1,2}\int_{\O_i} \phi_i (u(x,t)-v(x,s))^+\partial_t 
\psi(x,t)\rho_m(t-s)dxdtds\\
\ds + \int_0^T\sum_{i=1,2}\int_{\O_i} \phi_i (u_0(x)-v(x,s))^+\psi(x,0)
\rho_m(-s)dxds \\
\ds+\int_0^T\sum_{i=1,2}\int_{\O_i} \phi_i (u(x,t)-v_0(x))^+\psi(x,t)
\rho_m(t)dxdt\\
\ds - \int_0^T\int_0^T\sum_{i=1,2}\int_{\O_i} \partial_x 
(\phii(u)(x,t)-\phii(v)(x,s))^+\partial_x\psi(x,t)\rho_m(t-s)dxdtds\ge0.
\end{array}
\end{equation}
We can now let $m$ tend toward $+\infty$ in~(\ref{avant_cont}), and using  (\ref{ess-cont}) for $u$ and $v$, and  the
theorem of continuity in mean, we get: for all $\psi\in \Dd^+(\overline\O\times[0,T[)$ such that $\psi(0,t)=0$,
\begin{equation}\label{entro+}
\begin{array}{c}
\ds \int_0^T\sum_{i=1,2}\int_{\O_i} \phi_i (u(x,t)-v(x,t))^+\partial_t \psi(x,t)dxdt\\
\ds + \sum_{i=1,2}\int_{\O_i} \phi_i (u_0(x)-v_0(x))^+ \psi(x,0)dx \\
\ds - \int_0^T\sum_{i=1,2}\int_{\O_i} \partial_x (\phii(u)(x,t)-\phii(v)(x,t))^+
\partial_x\psi(x,t)dxdt\ge0.
\end{array}
\end{equation}

We aim now to extend the inequality~(\ref{entro+}) in the case where $\psi(0,t)\neq0$, and particularly in the case $\psi(x,t)=\theta(t)$, 
so that the third term disappears in~(\ref{entro+}).

To this purpose, let us  set here  $u_i(t)= u_i(0,t)$ to denote the trace of $u_i$ at the interface $\G$ (and correspondingly, $v_i(t)=v_i(0,t)$). 
We  introduce the subsets of $(0,T)$:
\begin{itemize}
\item $E_{u>v}=\{t\in[0,T]\  |\  u_1(t)>v_1(t) \text{ or } u_2(t)>v_2(t)\}$,
\item $E_{u \le v}=\{t\in[0,T]\  |\  u_1(t)\le v_1(t) \text{ and } u_2(t)\le v_2(t)\}$,
\end{itemize}
so that $E_{u \le v}$ is the complement of $E_{u>v}$ in $[0,T]$.

For all $\eps>0$, one defines $\psi_\eps(x)=\max\left(1-\frac{|x|}{\eps},0\right)$.
For all $\theta\in \Dd^+([0,T[)$, we take $(x,t)\mapsto \theta(t)
(1-\psi_\eps(x))$ instead of $\psi(x,t)$ as test-function in~(\ref{entro+}), 
thus we get:
\begin{equation*}
\begin{array}{c}
\ds \int_0^T\sum_{i=1,2}\int_{\O_i} \phi_i (u(x,t)-v(x,t))^+\partial_t 
\theta(t)(1-\psi_\eps(x))dxdt\\
\ds + \sum_{i=1,2}\int_{\O_i} \phi_i (u_0(x)-v_0(x))^+ 
(1-\psi_\eps)(x)\theta(0)dx \\
\ds -
\int_0^T\frac{\theta(t)}{\eps} \left(
\begin{array}{c}
 (\varphi_1(u)(-\eps,t)-\varphi_1(v)(-\eps,t))^+ -  
 (\varphi_1(u_1)(t)-\varphi_1(v_1)(t))^+\\
 + (\varphi_2(u)(\eps,t)-\varphi_2(v)(\eps,t))^+-  
 (\varphi_2(u_2)(t)-\varphi_2(v_2)(t))^+
 \end{array}
 \right)
  dt\ge0.
\end{array}
\end{equation*}
For almost every $t\in E_{u\le v}$, 
the function $(\phii(u)-\phii(v))^+(\cdot,t)$
admits a nil trace on $\{x=0\}$, thus the third term in the previous 
inequality can be reduced to the set $E_{u> v}$ obtaining
\begin{equation}\label{entro+1}
\begin{array}{c}
\ds \int_0^T \sum_{i=1,2}\int_{\O_i} \phi_i (u(x,t)-v(x,t))^+\partial_t 
\theta(t)(1-\psi_\eps(x))dxdt\\
\ds + \sum_{i=1,2}\int_{\O_i} \phi_i (u_0(x)-v_0(x))^+ (1-\psi_\eps)(x)\theta(0)dx \\
\ds +
\int_{E_{u>v}}{\theta(t)} \sum_{i=1,2}\int_{\O_i} \partial_x (\phii(u)(x,t)-
\phii(v)(x,t))^+\partial_x\psi_\eps(x) dxdt\ge0.
\end{array}
\end{equation}

%lemme clef
We  show now the crucial point of the uniqueness proof, which is the 
subject of the following lemma.
\begin{lem}\label{point_clef} For all $\theta\in\Dd^+([0,T[)$, if $u,v$ are both 
bounded flux solutions, i.e. if one has $\partial_x\phii(u),\partial_x\phii(v)\in 
L^\infty(\Qq_i)$ one has,
$$
\ds\limsup_{\eps\rightarrow 0} \int_{E_{u>v}}{\theta(t)} \sum_{i=1,2}\int_{\O_i} 
\partial_x (\phii(u)(x,t)-\phii(v)(x,t))^+\partial_x\psi_\eps(x) dxdt\le0.
$$
\end{lem}

Using the weak formulation~(\ref{weak_for_1d}), we can claim that for any regular function
$\vartheta\in\Dd([0,T[)$, 
\begin{equation}\label{vie}
\lim_{\eps\rightarrow 0} \int_0^T \vartheta(t)\sum_{i=1,2}\int_{\O_i}
\partial_x (\phii(u)-\phii(v))\partial_x\psi_\eps(x)dxdt=0.
\end{equation}
Since for $i=1,2$, $\partial_x (\phii(u)-\phii(v))$ belongs to $L^\infty(\OiT)$,
one has
$$
\left|\int_0^T \vartheta(t)\sum_{i=1,2}\int_{\O_i}
\partial_x (\phii(u)-\phii(v))\partial_x\psi_\eps(x)dxdt\right|\le
C\|\vartheta\|_{L^1(0,T)},
$$
then a density argument allows us to claim that (\ref{vie}) still holds for any
$\vartheta\in L^1(0,T)$, and particularly for
$\vartheta(t)=\theta(t)\1_{E_{u>v}}(t)$. Thus there exists $A(\eps)$ tending to $0$
as $\eps$ tends to $0$ such that
\begin{equation}\label{entro_trace2}
  \ds  \int_{E_{u>v}}\theta(t)\sum_{i=1,2}  \int_{\O_i} \partial_x(\phii(u)(x,t)-\phii(v)(x,t))\partial_x \psi_\eps(x) dx dt 
 =A(\eps).
 \end{equation} 
 
 Splitting up the positive and negative parts of $(\phii(u)(x,t)-\phii(v)
 (x,t))$, (\ref{entro_trace2}) becomes:
 \begin{equation}\label{entro_trace3}
\begin{array}{c}  \ds \int_{E_{u>v}}\theta(t) \sum_{i=1,2}  \int_{\O_i} 
\partial_x(\phii(u)(x,t)-\phii(v)(x,t))^+\partial_x \psi_\eps(x) dx dt \\
\ds
 =
 \int_{E_{u>v}}\theta(t)\sum_{i=1,2}  \int_{\O_i} 
 \partial_x(\phii(u)(x,t)-\phii(v)(x,t))^-\partial_x \psi_\eps(x) dx dt 
 +A(\eps).
 \end{array}
 \end{equation}
 It is at this point that we actually  use the monotony of the transmission condition, i.e. condition 3 in Definition \ref{gws}. Indeed, the  conditions $\t\pi_1(u_1(t))\cap\t\pi_2(u_2(t))\neq\emptyset$ and 
$\t\pi_1(v_1(t))\cap\t\pi_2(v_2(t))\neq\emptyset$ insure that :
\begin{equation}\label{monotony}
u_1>v_1 \implies u_2\ge v_2\qquad  \hbox{and }\qquad u_1<v_1 \implies u_2\le v_2\,.
\end{equation}
Therefore,  recalling the definition of the set $E_{u>v}$ and of $\psi_\eps$, the first term in the  right member  of~(\ref{entro_trace3}) is non-positive,  
 and then we conclude
$$
\ds\limsup_{\eps\rightarrow 0} \int_{E_{u>v}}{\theta(t)} \sum_{i=1,2}\int_{\O_i} 
\partial_x (\phii(u)(x,t)-\phii(v)(x,t))^+\partial_x\psi_\eps(x) dxdt\le0.
$$
This achieves the proof of lemma~\ref{point_clef}, and allows us to take the 
limit in inequality (\ref{entro+1}) for $\eps\rightarrow 0$. 
 Then for all $\psi\in\Dd^+([0,T[)$, one gets
 \begin{equation}\label{fin}
 -\int_0^T \sum_{i=1,2}\int_{\O_i} \phi_i (u(x,t)-v(x,t))^+\partial_t \psi(t)dxdt
 \le \sum_{i=1,2}\int_{\O_i} \phi_i (u_0(x)-v_0(x))^+\psi(0)dx .
 \end{equation}
One can also prove exactly the same way that
\begin{equation}\label{fin2}
-\int_0^T \sum_{i=1,2}\int_{\O_i} \phi_i (u(x,t)-v(x,t))^-\partial_t \psi(t)dxdt
 \le \sum_{i=1,2}\int_{\O_i} \phi_i (u_0(x)-v_0(x))^-\psi(0)dx .
\end{equation}
These inequalities still hold for $\psi=(T-t)$, and then if $u_0=v_0$, one has $u=v$
almost everywhere in $\Qq$. 
Moreover  we can take $\psi(t) =\1_{[0,s]} (t)$ as test function in (\ref{fin}) to get 
the $L^1$-contraction principle~(\ref{contract}) stated in theorem~\ref{prop_unicite}.
\end{proof}

In the sequel, we prove that for any $u_0$ in $L^\infty(-1;1)$, $0\le u_0 \le 1$, 
there exists a unique weak solution of problem~(\ref{P}) which is the 
limit of a sequence of bounded flux solutions $(u_n)_n$, i.e. for all $n\ge 1$, 
$\partial_x\phii(u_n)\in L^\infty(\Qq_i)$.

\begin{thm}[Existence and uniqueness of the SOLA]\label{existence_SOLA}
Let $u_0\in L^\infty(-1,1)$, $0\le u_0\le 1$, and let $(u_{0,n})_{n\ge1}$ be a 
sequence of bounded flux initial data, i.e. for all $n\ge 1$,
\begin{itemize}
\item $0\le u_{0,n}\le 1$,
\item $\phii(u_{0,n})\in W^{1,\infty}(\O_i)$,
\item $\t\pi_1(u_{0,n,1})\cap\t\pi_2(u_{0,n,2})\neq\emptyset,$
\end{itemize}
 such that $$\lim_{n\rightarrow+\infty} \|u_{0,n}-u_0\|_{L^1(\O)} =0. $$
Let $(u_n)_{n\ge1}$ be the sequence of the bounded flux  solutions to the problem~(\ref{P}) 
for $u_{0,n}$ as initial data. Then the  sequence $(u_n)_{n\ge1}$ converges toward 
$u$ in $C(]0,T[,L^p(-1,1))$, $1\le p<+\infty$, where $u$ is a 
solution to the problem~(\ref{P}), called Solution Obtained as  Limit of Approximation (SOLA). 
Furthermore, if $u,v$ are two SOLAs, 
for initial data $u_0,v_0$, one has the following $L^1$-contraction 
principle: $\forall t\in [0,T]$, 
\begin{equation}
\SiN\int_{\O_i} \phi_i (u(x,t)-v(x,t))^\pm dx \le \SiN  \int_{\O_i} \phi_i 
(u_{0}(x)-v_{0}(x))^\pm dx.
\end{equation}
This particularly leads to the uniqueness of the SOLA.
\end{thm}

\begin{proof}
Let $(u_{0,n})$ be a regular sequence of initial data converging toward $u_0$ in 
$L^1(-1,1)$ - one take e.g. $u_{0,n}\in C^\infty_c (]-1,0[\cup]0,1[)$. 
Then $(u_{0,n})$ is a Cauchy sequence,
and thanks to~(\ref{contract}), for all $t\in[0,T]$, 
$$
 \SiN\int_{\O_i} \phi_i |u_n(x,t)-u_m(x,t)| dx \le \SiN  \int_{\O_i} \phi_i 
|u_{0,n}(x)-u_{0,m}(x)| dx.
$$
Thus $(u_n)_n$ is a Cauchy sequence in $C([0,T];L^1(\O))$  and  converges to a function $u$  
in $C([0,T];L^1(\O))$. Since $(u_n)_n$ is bounded in $L^\infty(\Qq)$, 
one has $u_n\!\!\rightarrow\! u$ in $C([0,T];L^p(-1,1))$.

We now have to check that $u$ is a  weak solution to 
the problem~(\ref{P}). It is easy to check, using to the 
$L^\infty$-bound of $u_n$,  
that $\phii(u_n)$ tends toward $\phii(u)$ in $L^p(\OiT)$, 
for all $p\in[1,+\infty[$.
Thanks to~(\ref{L2H1_estimate}), the sequence $(\phii(u_n))_n$ is bounded in 
$L^2(0,T;H^1(\O_i))$, and thus $\phii(u_n)\rightarrow\phii(u)$ 
weakly in $L^2(0,T;H^1(\O_i))$, and $\phii(u_n)$ converges  
in $L^2(0,T;H^s(\O_i))$, for all $s\in]0,1[$, still toward $\phii(u)$.
Particularly, $u_{n,i}(t)$ tends toward $u_i(t)$. Since the set 
$\{(a,b)\in [0,1]^2\ |$ $\t\pi_1(a)\cap\t\pi_2(b)\neq\emptyset \}$ is closed,  
we can claim that 
$$\t\pi_1(u_1(t))\cap\t\pi_2(u_2(t))\neq\emptyset\quad \text{ for a.e. }t\in[0,T].$$
We can also pass to the limit in the weak formulation in order to conclude 
that $u$ is a  weak solution to the problem~(\ref{P}), achieving this way the existence of a SOLA $u$.

Let now $v$ be another SOLA, obtained through  a sequence $(v_{0,n})_n$ of regular initial data 
converging toward $v_0$. Thanks to~(\ref{contract}), one has,
$$
\SiN\int_{\O_i} \phi_i |u_n(x,t)-v_n(x,t)| dx \le \SiN  \int_{\O_i} \phi_i 
|u_{0,n}(x)-v_{0,n}(x)| dx,
$$
whose limit as n tends toward $+\infty$ gives the attempted $L^1$-contraction 
principle:
$$
\SiN\int_{\O_i} \phi_i |u(x,t)-v(x,t)| dx \le \SiN  \int_{\O_i} \phi_i 
|u_{0}(x)-v_{0}(x)| dx,
$$
and so the uniqueness of the SOLA, completing 
the proof of theorem~\ref{existence_SOLA}.
\end{proof}

%=========================BIBLIOGRAPHIE====================================
%

%
%\bibliographystyle{alpha}
%\bibliography{../../../ccances}

\end{document}